\documentclass[11pt, reqno]{amsart}
\usepackage{lineno,hyperref}
\usepackage{cases}
\usepackage[square, comma, sort&compress, numbers]{natbib}%
\usepackage{float}
\usepackage{pifont}
\usepackage{bbding}
\usepackage{amssymb}
\usepackage{mathrsfs}
\usepackage{subfigure}
\usepackage{caption}
\usepackage{geometry}
\usepackage{color}
\usepackage{amsmath, amsthm, amscd, amsfonts, amssymb, graphicx, color}
\usepackage{lineno}
\usepackage{graphicx}
\usepackage{subfigure}
\newtheorem{theorem}{Theorem}
\newtheorem{thm}{Theorem}
\newtheorem{lemma}{Lemma}

\newtheorem{corollary}{Corollary}

\newtheorem{remark}{Remark}

\numberwithin{equation}{section}
\numberwithin{equation}{section}
\numberwithin{dfn}{section}
\numberwithin{lemma}{section}
\numberwithin{theorem}{section}
\numberwithin{thm}{section}
\numberwithin{corollary}{section}
\numberwithin{prop}{section}
\numberwithin{remark}{section}

\newcommand{\abs}[1]{\left\vert#1\right\vert}

\date{\today}

\begin{document}
\setcounter{page}{1}

\title[On quasiconformal extensions of harmonic mappings]
{On quasiconformal extensions of harmonic mappings associated with pre-Schwarzian derivative}

\author[X.-Y. Wang, J.-H. Fan,  Z.-Y. Hu and Z.-G. Wang]{Xiao-Yuan Wang, Jin-Hua Fan,  Zhen-Yong Hu  and Zhi-Gang Wang$^{*}$}

\vskip.10in
\address{\noindent Xiao-Yuan Wang \vskip.05in
 School of Science, Nanjing University of Science and Technology,
 Nanjing 210094, Jiangsu, P. R.
China.}
%
\email{\textcolor[rgb]{0.00,0.00,0.84}{mewangxiaoyuan$@$163.com}}

\address{\noindent Jin-Hua Fan\vskip.05in
 School of Science, Nanjing University of Science and Technology,
 Nanjing 210094, Jiangsu, P. R. China.}\vskip.05in
\email{\textcolor[rgb]{0.00,0.00,0.84}{jinhuafan$@$hotmail.com}}

\address{\noindent Zhen-Yong Hu\vskip.05in
 School of Science, Nanjing University of Science and Technology,
 Nanjing 210094, Jiangsu, P. R. China.}\vskip.05in
\email{\textcolor[rgb]{0.00,0.00,0.84}{huzhenyongad$@$163.com}}

\address{\noindent Zhi-Gang Wang\vskip.05in
School of Mathematics and Statistics, Hunan
	First Normal University, Changsha 410205, Hunan, P. R. China.}
\vskip.05in
\email{\textcolor[rgb]{0.00,0.00,0.84}{wangmath$@$163.com}}
\thanks{$^*$Corresponding author.}



\date{\today}

\begin{abstract}
In this paper, we extend Ahlfors's univalent criteria and Ahlfors's quasiconformal extension for analytic functions to harmonic mappings defined in the unit disk. Moreover, we give
a general quasiconformal extension of harmonic Teichm\"{u}ller mappings, whose maximal dilatation estimate is
asymptotically sharp.
\end{abstract}
\maketitle

\section{Introduction and statements of the main results}\label{sec-1}
Let $\triangle:=\{z:|z|<1\}$ be the unit disk in the   complex plane $\mathbb{C}$. For a locally univalent analytic function  $\phi$ in $\triangle$, we denote by $P_\phi$ the pre-Schwarzian derivative  and $S_{\phi}$ the Schwarzian derivative of $\phi$ as follows:
\begin{equation*}
\begin{aligned}
& P_\phi=\frac{\phi^{\prime \prime}}{\phi^{\prime}}\ \ {\rm and}\  \ S_{\phi}=\left(P_{\phi}\right)^{\prime}-\frac{1}{2}\left(P_{\phi}\right)^{2}.
\end{aligned}
\end{equation*}
For  $P_\phi$ and $S_{\phi}$, we have the following two norms:
$$
\|P_{\phi}\|=\sup _{z \in \triangle}|P_{\phi}| (1-|z|^2),
$$
and
$$
\left\|S_{\phi}\right\|=\sup _{z \in \triangle}\left|S_{\phi}\right| (1-|z|^2)^2.
$$

In 1972, Becker \cite{Becker-1972} stated that if
\begin{equation}\label{eq-1.01}
 \|P_{\phi}\| \leq 1,
\end{equation}
then $\phi$ is univalent in $\triangle$. Moreover, the constant 1 is sharp (see \cite{Becker-1984}). Becker \cite{Becker-1972} also proved that if
\begin{equation}\label{eq-1.02}
\|P_{\phi}\| \leq k<1,
\end{equation}
then not only $\phi$ is univalent, but also it exists a continuous extension $\widetilde{\phi}$ to $\overline{\triangle}$ and $\widetilde{\phi}(\partial \triangle)$ is a quasicircle.

Indeed, Becker \cite{Becker-1972} proved that $\phi$ has a $K$-quasiconformal extension to   $\mathbb{C}$ if \eqref{eq-1.02} holds. We call \eqref{eq-1.01} and \eqref{eq-1.02} are Becker's univalence criterion and Becker's extended univalence criterion, respectively.
In 1974, Ahlfors gave a generalized  result  expressed as follows.
\begin{thm}\label{thm-Ahlfors-1974-0}
Let $\phi$ be a locally univalent analytic function in $\triangle$.  $\sigma$ is a continuous function in $\triangle$, which satisfies the following conditions: (i) $\sigma_{z}$ and $\sigma_{\overline{z}}$ exist in $\triangle$ a.e.; (ii) $\frac{1}{\sigma}=0$ on $\partial \triangle$; (iii) $\frac{\sigma_{\overline{z}}}{\sigma^{2}} \neq 0$ in $\overline{\triangle}$. Then the inequality
\begin{equation}\label{eq-Ahlfors-1974-0}
\left|\sigma P_{\phi}+\sigma^{2}-\sigma_{z}\right| \leq  \left|\sigma_{\overline{z}}\right| \ \ (z\in \triangle)
\end{equation}
is sufficient to imply that $\phi$ is univalent in $\triangle$.
\end{thm}

Moreover, Ahlfors \cite{Ahlfors-1974} gave a generalized  quasiconformal extension formulated as below.

\begin{thm}\label{thm-Ahlfors-1974-1}
Let $\phi$ be a locally univalent analytic function in $\triangle$.  $\sigma$ is a continuous function in $\triangle$, which satisfies conditions given by Theorem \ref{thm-Ahlfors-1974-0}. Then the inequality
\begin{equation}\label{eq-Ahlfors-1974-1}
\left|\sigma P_{\phi}+\sigma^{2}-\sigma_{z}\right| \leq k\left|\sigma_{\overline{z}}\right| \ \  (0 \leq  k<1)
\end{equation}
is sufficient to imply that $\phi$ has  an explicit homeomorphic extension
\begin{equation*}
\Phi(z)=
\begin{cases}
\widetilde{\phi}(z) & (|z| \leq 1), \\
\phi\left(\frac{1}{\overline{z}}\right)+ u\left(\frac{1}{\overline{z}}\right) & (|z|>1),
\end{cases}
\end{equation*}
where $u(z)=\phi^{\prime}(z)/ {\sigma(z)}$ for $z \in \triangle \backslash\{0\}$. Also, the mapping $\Phi$ is $K$-quasiconformal in $\mathbb{C}$, where $K=\frac{1+k}{1-k}$.
\end{thm}

A special case of interest, suggested by Ahlfors \cite{Ahlfors-1974}, arises from taking
\begin{equation}\label{eq-Ahlfors-c+1}
\sigma(z)=\frac{(c+1)\overline{z}}{1-|z|^2}
\end{equation}
in \eqref{eq-Ahlfors-1974-1}. Then the principal condition for  quasiconformal extension displayed as follows:
\begin{thm}\label{thm-Ahlfors-1974-2}
If
\begin{equation}\label{eq-Ahlfors-1974-3}
\left|c| z|^{2}+\left(1-|z|^{2}\right) z P_{\phi} \right| \leq k<1  \quad (|c| \leq k;\, z \in \triangle),
\end{equation}
then $\phi$ has an explicit homeomorphic extension
\begin{equation}\label{eq-Ahlfors-1974-4}
\Phi(z)=
\begin{cases}\widetilde{\phi}(z) & (|z| \leq 1), \\
\phi\left(\frac{1}{\overline{z}}\right)+ u\left(\frac{1}{\overline{z}}\right) &(|z|>1),
\end{cases}
\end{equation}
where $$u(z)=\frac{\phi^{\prime}(z)\left(1-|z|^{2}\right)}{\overline{z}}\quad (z \in \triangle\backslash\{0\}).$$  Also, the mapping $\Phi$ is a $K$-quasiconformal extension in $\mathbb{C}$, where $K=\frac{1+k}{1-k}$.
\end{thm}

We remark that even though \eqref{eq-Ahlfors-1974-1}  appears to depend on $\phi$ only through
$P_{\phi}$, one may, by choosing $\sigma$  to depend on $\phi$, obtain conditions such as the classical
Becker criterion which involve $\phi$ through quantities other than $P_{\phi}$.
Instead of  $P_{\phi}$ in Theorem \ref{thm-Ahlfors-1974-1}, Nehari \cite{Nehari-1949} used non-oscillating solutions of differential equations, established the following  criterion with Schwarzian derivative.
\begin{thm}\label{thm-Nehari}
For a locally univalent analytic function $\phi$ in  $\triangle$, the condition
\begin{equation} \label{eq-Nehari}
\left\|S_{\phi}(z)\right\| \leq 2
\end{equation}
 implies that $\phi$ is univalent.
\end{thm}
We call  Theorem \ref{thm-Nehari} is the classical  Nehari's univalence criterion.
At the same time, as a supplement note, Hille showed that the bound 2 is the best possible.
Later, a remarkable feature of  the explicit extension  was provided  by Ahlfors and Weill \cite{Ahlfors-Weill-1962}.

\begin{thm}\label{thm-Ahlfors-Weill}
Let $\phi$ be a locally univalent analytic function in $\triangle$. For a  constant $t\;(0\leq t<1)$,  if $\|S_{\phi}\|\leq 2t$, then the mapping
\begin{equation} \label{eq-Ahlfors-Weill}
F(z)  \,=\, \left\{
\begin{array}{rl}
\phi(z)   &\quad (|z|\leq1),\\
E_\phi(1/\overline{z}) &  \quad (|z|>1),
\end{array} \right.
\end{equation} is a quasiconformal extension of $\phi$ to $\mathbb{C}$,
where
$$
E_{\phi}(\xi)  = \phi(\xi)+  \frac{(1-|\xi|^2)\phi^{\prime}(\xi)}{\overline{\xi}-\tfrac{1}{2}(1-|\xi|^2) P_{\phi}(\xi)}   \quad (\xi \in\triangle).
$$
\end{thm}

\begin{remark}
{\rm  One can see that Theorem \ref{thm-Ahlfors-Weill} can be obtained by taking $$\sigma(z)= \frac{\overline{z}}{1-|z|^2}-\frac{1}{2}P_{\phi}$$ in Theorem \ref{thm-Ahlfors-1974-1}. Indeed, Ahlfors and Weill \cite{Ahlfors-Weill-1962} applied a quasiconformal extension to obtain the Nehari condition \eqref{eq-Nehari} for the half plane.}
\end{remark}

Ahlfors \cite{Ahlfors-1974} also showed  general formulas on the problem of univalence and quasiconformal extensibility  through the following two results.

\begin{thm}\label{thm-Ahlfors-1974-3}
Let $\phi$ be a locally univalent analytic function in $\triangle$. Then  the inequality
\begin{equation}\label{eq-Ahlfors-1974-5}
\left|\frac{1}{2} S_{\phi}+v^{2}-v_{z}\right| \leq k\left|v_{\overline{z}}\right| \ \ (0<k<1)
\end{equation}
is sufficient to imply that there exists a $\frac{1+k}{1-k}$-quasiconformal extension of $\phi$ to $\mathbb{C}$, where the continue  function $v$ satisfies  $\frac{v_{\overline{z}}}{v^{2}} \neq 0$ and $v \rightarrow \infty$ for $|z| \rightarrow 1$.
\end{thm}

%
%

\begin{thm}\label{thm-Ahlfors-1974-4}
Let $\phi$ be a locally univalent analytic function in $\triangle$. Then the inequality
\begin{equation}\label{eq-1.09}
\begin{aligned}
\left|\frac{1}{2} S_{\phi}(1-|z|^2)^2- c(1-c) \overline{z} \right|  \leq k|c| \quad (0<k<1;\, |c-1|\leq k)
\end{aligned}
\end{equation}
is  sufficient to imply that there exists a $\frac{1+k}{1-k}$-quasiconformal extension of $\phi$ to $\mathbb{C}$.
\end{thm}

We call \eqref{eq-Ahlfors-1974-0} is the Ahlfors's  univalent criterion  and  \eqref{eq-Ahlfors-1974-1},  \eqref{eq-Ahlfors-1974-3}, \eqref{eq-Ahlfors-1974-5},  \eqref{eq-1.09} are  the     Ahlfors's  quasiconformal extension criteria. Note that the   Becker's extension criterion arises from taking
\begin{equation}\label{eq-Ahlfors-c=0}
\sigma(z)=\frac{\overline{z}}{1-|z|^2}
\end{equation}
in \eqref{eq-Ahlfors-1974-1} or $c=0$ in \eqref{eq-Ahlfors-1974-3}.

\begin{remark}
{\rm By setting $\sigma(z)= v-1/2P_{\phi}$  and $\sigma(z)= c\overline{z}/(1-|z|^2)-1/2P_{\phi}$ in \eqref{eq-Ahlfors-1974-1}, one can get Theorem \ref{thm-Ahlfors-1974-3} and  Theorem \ref{thm-Ahlfors-1974-4}, respectively.}
\end{remark}

\begin{remark}
{\rm By putting  $v(z)=c\overline{z}/(1-|z|^2)$ in \eqref{eq-Ahlfors-1974-5}, one can get Theorem \ref{thm-Ahlfors-1974-4}. }
\end{remark}
\begin{remark}
{\rm By taking $v(z)= \overline{z}/(1-|z|^2)$ in \eqref{eq-Ahlfors-1974-5} or $c=1$ in \eqref{eq-1.09}, one can obtain Theorem \ref{thm-Ahlfors-Weill}.}
\end{remark}

A complex-valued function $f$ in $\triangle$ is harmonic if $\Delta f=4 f_{z \overline{z}}=0$. It is well-known that $f$ has a canonical representation $f=h+\overline{g}$, where $h$ and $g$ are analytic in $\triangle$. Lewy \cite{Lewy-1936} proved that a harmonic mapping $f$ is locally univalent if and only if its Jacobian $J_{f}=\left|h^{\prime}\right|^{2}-\left|g^{\prime}\right|^{2} \neq 0$. If $J_{f}>0\, \left(J_{f}<0\right)$, then $f$ is sense-preserving (sense-reserving). Denote by $\omega={g^{\prime}}/{h^{\prime}}$ the second complex dilatation of $f=h+\overline{g}$.
If a harmonic mapping has the representation $f=h+\alpha \overline{h}$, where $h$ is a conformal mapping and $\alpha$ is a constant such that $0<|\alpha|<1$, then it is called a harmonic Teichm\"{u}ller mapping (see \cite{Chen-Fang-2006}).

As natural generalizations of conformal mappings, quasiconformal extensions for harmonic mappings have been investigated by many researchers (see \cite{Efraimidis-2021, Efraimidis-Hernandez-Martin-2021} and the references therein).

In 2003, Chuaqui $et \ al.$ \cite{Chuaqui-Duren-Osgood-2003} generalized the idea of Schwarzian derivatives of $f=h+\overline{g}$ into the case of locally univalent harmonic mappings with $\omega=q^{2}$, where $q$ is an analytic function. The pre-Schwarzian derivative $P_f$ and Schwarzian derivative $S_f$ for sense-preserving harmonic mappings are defined as follows:
\begin{align}\label{eq-Pf}
P_f=\frac{\partial}{\partial z} \log \left|J_{f}\right|=P_{h}-\frac{\overline{\omega} \omega^{\prime}}{1-|\omega|^{2}},\end{align}   and
\begin{align}\label{eq-Pf1}S_f=S_{h}+\frac{\overline{\omega}}{1-|\omega|^{2}}\left(\frac{h^{\prime \prime}}{h^{\prime}} \omega^{\prime}-\omega^{\prime \prime}\right)-\frac{3}{2}\left(\frac{\overline{\omega} \omega^{\prime}}{1-|\omega|^{2}}\right)^{2},\end{align}
which were introduced and investigated by  Hern\'{a}ndez and   Mart\'{\i}n \cite{Hernandez-Martin-2015}. In the same paper, Hern\'{a}ndez and   Mart\'{\i}n proved following result.
\begin{thm}\label{thm-Hernandez-Martin-2015-1}
Let $f=h+\overline{g}$ be a sense-preserving harmonic mapping in $\triangle$ with the second dilatation $\omega$. If $f$ satisfies
\begin{equation}\label{eq-Hernandez-Martin-2015-1}
\left(1-|z|^{2}\right)|P_f|+\frac{\left|\omega^{\prime}(z)\left(1-|z|^{2}\right)\right|}{1-|\omega(z)|^{2}} \leq 1 \ \ (z \in \triangle),
\end{equation}
then $f$ is univalent in  $\triangle$. The constant 1 is sharp.
\end{thm}

Moreover, Hern\'{a}ndez and   Mart\'{\i}n \cite{Hernandez-Martin-2013} obtained a quasiconformal extension as follows.
\begin{thm}\label{thm-Hernandez-Martin-2013-1}
Let $f=h+\overline{g}$ be a sense-preserving harmonic mapping in $\triangle$ with $\|\omega\|_{\infty}<1$. If $f$ satisfies the condition
\begin{equation}\label{eq-Hernandez-Martin-2013-1}
\left(1-|z|^{2}\right)|P_f|+\frac{\left|\omega^{\prime}(z)\left(1-|z|^{2}\right)\right|}{1-|\omega(z)|^{2}} \leq k<1 \quad (z \in \triangle),
\end{equation}
then $f$ has a continuous and injective extension $\widetilde{f}$ to $\overline{\triangle}$, and the mapping
\begin{equation}\label{eq-Hernandez-Martin-2013-2}
F(z)= \begin{cases}
\widetilde{f}(z) & (|z| \leq 1),
\\ f\left(\frac{1}{\overline{z}}\right)+U\left(\frac{1}{\overline{z}}\right) & (|z|>1),
\end{cases}
\end{equation}
is a homeomorphism of $\mathbb{C}$ onto itself, where
$$
U(z)=\frac{h^{\prime}(z)}{\overline{z}}\left(1-|z|^{2}\right)+\frac{\overline{g^{\prime}(z)}}{z}\left(1-|z|^{2}\right)  \quad (z \in \triangle \backslash\{0\}).
$$
Moreover, if $f$ satisfies the condition \eqref{eq-Hernandez-Martin-2013-1}, then $\widetilde{f}(\partial \triangle)$ is a quasicircle and $f$ can be extended to a quasiconformal mapping in $\mathbb{C}$. Under the additional condition
\begin{equation}\label{eq-Hernandez-Martin-2013-3}
k<\frac{1-\|\omega\|_{\infty}}{1+\|\omega\|_{\infty}},
\end{equation}
the mapping $F$ defined by \eqref{eq-Hernandez-Martin-2013-2} is an explicit $K$-quasiconformal mapping of $\mathbb{C}$ onto itself, where
$$
K=\frac{1+k+\|\omega\|_{\infty}(1-k)}{1-k-\|\omega\|_{\infty}(1+k)}.
$$
\end{thm}

The following result was derived by Bravo $et \ al.$ \cite{Bravo-Hernandez-Venegas-2017}, which generalized Theorem \ref{thm-Hernandez-Martin-2015-1}. It also can be regarded  as a generalization  of harmonic analogue of Theorem \ref{thm-Ahlfors-1974-0}.

\begin{thm}\label{thm-Bravo-Hernandez-Venegas-2017}
Let $f=h+\overline{g}$ be a sense-preserving harmonic mapping in $\triangle$ with the second dilatation $\omega$. If $f$ satisfies
\begin{equation}\label{eq-thm-Bravo-Hernandez-Venegas-2017}
\abs{c| z |^{2}+(1-|z|^{2}) z {P_f}}+\frac{|z \omega^{\prime}(z) (1-|z|^{2})|}{1-|\omega(z)|^{2}} \leq 1 \ \ ( z \in \triangle),
 \end{equation}
then $f$ is univalent in  $\triangle$.
\end{thm}

Recently, Chen and Que \cite{Chen-Que-2017} generalized Theorem \ref{thm-Hernandez-Martin-2013-1} as follows.
\begin{thm}\label{thm-Chen-Que-2017-1}
Let $f=h+\overline{g}$ be a sense-preserving harmonic mapping in $\triangle$ with $\|\omega\|_{\infty}<1$. If $f$ satisfies \eqref{eq-Hernandez-Martin-2013-1}, then for all $|\lambda| \leq 1, f_{\lambda}=h+\lambda \overline{g}$ has a continuous and injective extension $\widetilde{f_{\lambda}}$ to $\overline{\triangle}$ and the mapping
\begin{equation}\label{eq-Chen-Que-2017-1}
F_{\lambda}(z)=
\begin{cases}\widetilde{f}_{\lambda}(z)  &  (|z| \leq 1), \\
f_{\lambda}\left(\frac{1}{\bar{z}}\right)+U_{\lambda}\left(\frac{1}{\overline{z}}\right)  & (|z|>1),
\end{cases}
\end{equation}
is a homeomorphism of $\mathbb{C}$ onto itself, where
$$
U_{\lambda}(z)=\frac{h^{\prime}(z)}{\overline{z}}\left(1-|z|^{2}\right)+\lambda \frac{\overline{g^{\prime}(z)}}{z}\left(1-|z|^{2}\right) \quad (z \in \triangle \backslash\{0\}).
$$
Moreover, if $k$ satisfies   \eqref{eq-Hernandez-Martin-2013-3}, then  the family of mappings $F_{\lambda}(z)$ are $K$-quasiconformal mappings in $\mathbb{C}$, where
\begin{equation}\label{eq-Chen-Que-2017-2}
K=\frac{1+k+|\lambda|\|\omega\|_{\infty}(1-k)}{1-k-|\lambda|\|\omega\|_{\infty}(1+k)}.
\end{equation}
\end{thm}

We observe that the following result obtained by Chen and Que  \cite{Chen-Que-2017} is a quasiconformal extension of  harmonic
Teichm\"{u}ller mappings, whose maximal dilatation estimate is asymptotically sharp.

\begin{thm}\label{thm-Chen-Que-2017-2}
Let $f$ be a sense-preserving harmonic mapping in  $\triangle$ with a representation $f=h+\alpha \overline{h}$, where $h$ is a locally univalent analytic function in $\triangle$ and $\alpha$ is a constant with $|\alpha|<1$. Assume that
\begin{equation}\label{eq-Chen-Que-2017-3}
\left|P_{h}(z)\left(1-|z|^{2}\right)\right| \leq k<1  \quad (z \in \triangle).
\end{equation}
Then $f$ is a harmonic Teichm\"{u}ller mappings of $\triangle$ and has a continuous and homeomorphic extension $\widetilde{f}$ to $\overline{\triangle}$. Moreover, the mapping
\begin{equation}\label{eq-Chen-Que-2017-4}
F(z)= \begin{cases}
\widetilde{f}(z) &  \quad  (|z| \leq 1), \\
f\left(\frac{1}{\overline{z}}\right)+U_{\alpha}\left(\frac{1}{\overline{z}}\right) &  \quad (|z|>1),
\end{cases}
\end{equation}
is a $K$-quasiconformal mapping of $\mathbb{C}$ with
\begin{equation}\label{eq-Chen-Que-2017-5}
K=\frac{(1+k)(1+|\alpha|)}{(1-k)(1-|\alpha|)},
\end{equation}
where
$$
U_{\alpha}(z)=\frac{h^{\prime}(z)\left(1-|z|^{2}\right)}{\bar{z}}+\frac{\alpha \overline{h^{\prime}(z)}\left(1-|z|^{2}\right)}{z} \quad (z \in \triangle \backslash\{0\}).
$$
The maximal dilatation estimate of the quasiconformal extension $F$ is asymptotically sharp in $k$, and extremal mappings are of the form $f(z)=a z+b \overline{z}$, where $a$ and $b$ are two nonvanishing constants.
\end{thm}

In a recent paper, Hu and Fan \cite{Hu-Fan-2021} generalized  Theorem \ref{thm-Chen-Que-2017-1} by posing one of   harmonic version of the classical Becker's criteria. We present it here in a slightly  modified form.

\begin{thm}\label{thm-Hu-Fan-2021}
 Let $f=h+\overline{g}$ be a sense-preserving harmonic mapping in $\triangle$. Assume that
 \begin{equation}\label{eq-Hu-Fan-2021-1}
\abs{c| z |^{2}+(1-|z|^{2}) z {P_f}}+\frac{|z \omega^{\prime}(z) (1-|z|^{2})|}{1-|\omega(z)|^{2}} \leq k<1 \quad(|c|\leq k;\, z \in \triangle).
 \end{equation}
Then $f_{\lambda}=h+\lambda \overline{g}\; (|\lambda| \leq 1)$ has a continuous and injective extension $\widetilde{f_{\lambda}}$ to $\overline{\triangle}$ and the family of mappings $F_{\lambda}(z)$ defined by \eqref{eq-Chen-Que-2017-1} are $K$-quasiconformal mappings of $\mathbb{C}$ onto themselves.
Moreover, if $k$ satisfies the conditions  \eqref{eq-Hernandez-Martin-2013-3}, then  the family of mappings $F_{\lambda}(z)$ are $K$-quasiconformal mappings of $\mathbb{C}$, where $K$ is given by   \eqref{eq-Chen-Que-2017-2}.
\end{thm}

\begin{remark}
{\rm  In Theorem \ref{thm-Chen-Que-2017-1},  the condition \eqref{eq-Hernandez-Martin-2013-1} for $P_f$ and \eqref{eq-Hernandez-Martin-2013-3}  for $k$, imply  that $\widetilde{f}_{\lambda}(\partial\triangle)$  is a quasicircle.
In Theorem \ref{thm-Hu-Fan-2021},  if inequality \eqref{eq-Hu-Fan-2021-1} holds for $P_f$, it shows that $\widetilde{f}_{\lambda}(\partial\triangle)$  is a quasicircle.
 We would like to draw
reader's attention to  Theorem \ref{thm-1.02}, it is shown that $\widetilde{f}_{\lambda}(\partial\triangle)$  is a quasicircle, if inequality \eqref{eq-1.05} holds  for $P_f$.}
\end{remark}

Motivated essentially by the above discussions, we aim at deriving general harmonic analogues of Theorems \ref{thm-Ahlfors-1974-0} and \ref{thm-Ahlfors-1974-1} by using the pre-Schwarzian derivative of harmonic mappings.

\begin{theorem}\label{thm-1.00}
Let $f=h+\overline{g}$ be a sense-preserving harmonic mapping in $\triangle$ and let $\omega$ be its second complex dilatation.
Let $\sigma$ be a continuous function in  $\triangle$, which satisfies the following conditions: (i) $\sigma_{z}$ and $\sigma_{\overline{z}}$ exist in $\triangle$ a.e.; (ii) $\frac{1}{\sigma}=0$ on $\partial \triangle$; (iii) $\frac{\sigma_{\overline{z}}}{\sigma^{2}} \neq 0$ in $\overline{\triangle}$. If $f$ satisfies the condition
\begin{equation}\label{eq-1.04}
\left|\sigma P_f+\sigma^{2}-\sigma_{z}\right| +  \frac{|\sigma \omega^{\prime}|}{1-|\omega|^2}  \leq  \left|\sigma_{\overline{z}}\right| \ \ (z \in \triangle),
\end{equation}
then $f$ is univalent in  $\triangle$.
\end{theorem}

\begin{remark}
{\rm By setting $\sigma(z)=\frac{(c+1)\overline{z}}{1-|z|^2}$ in \eqref{eq-1.04} of Theorem \ref{thm-1.00}, we get  Theorem \ref{thm-Bravo-Hernandez-Venegas-2017} due to Bravo $et \ al.$ \cite{Bravo-Hernandez-Venegas-2017};
By putting $\sigma(z)=\frac{\overline{z}}{1-|z|^2}$ in \eqref{eq-1.04} of Theorem \ref{thm-1.00}, we obtain Theorem \ref{thm-Hernandez-Martin-2015-1} due to Hern\'{a}ndez and   Mart\'{\i}n \cite{Hernandez-Martin-2015}.
}
\end{remark}

\begin{theorem}\label{thm-1.01}
Let $f=h+\overline{g}$ be a sense-preserving harmonic mapping in $\triangle$ and let $\omega$ be its second complex dilatation with $\|\omega\|_{\infty}<1$.
Let $\sigma$ be a continuous function in  $\triangle$, which satisfies   conditions given  in  Theorem \ref{thm-1.00}. Assume that
\begin{equation}\label{eq-1.05}
\left|\sigma P_f+\sigma^{2}-\sigma_{z}\right| +  \frac{|\sigma \omega^{\prime}|}{1-|\omega|^2}  \leq k\left|\sigma_{\overline{z}}\right| \ \ (0 \leq k<1;\, z \in \triangle).
\end{equation}
Then  the harmonic mapping $f_{\lambda}=h+\lambda \overline{g}\; (|\lambda|\leq 1)$ has a continuous and injective extension $\widetilde{f}$ to $\overline{\triangle}$. Moreover, the function
\begin{equation}\label{eq-1.06}
F_{\lambda}(z)= \begin{cases}
    \widetilde{f}_{\lambda}(z) &(|z| \leq 1), \\
    f_{\lambda}\left(\frac{1}{\overline{z}}\right)+U_{\lambda}\left(\frac{1}{\overline{z}}\right) &(|z|>1),
      \end{cases}
\end{equation}
is a homeomorphic extension of $f$ to $\mathbb{C}$ onto itself. The function $U_{\lambda}$ that appears in \eqref{eq-1.06} is defined by
$$
U_{\lambda}(z)=\frac{h^{\prime}(z)}{\sigma(z)}+\lambda\frac{\overline{g^{\prime}(z)}}{\overline{\sigma(z)}}  \quad (z \in \triangle \backslash\{0\}).
$$
Furthermore, if $k$ satisfies the condition  \eqref{eq-Hernandez-Martin-2013-3}, then  the family of mappings $F_{\lambda}(z)$ are  $K$-quasiconformal mappings of $\mathbb{C}$, where $K$ is given by \eqref{eq-Chen-Que-2017-2}.
\end{theorem}

\begin{remark}
{\rm By setting $\lambda=0$ in Theorem \ref{thm-1.01}, we get Theorem \ref{thm-Ahlfors-1974-1} exactly given  by Ahlfors \cite{Ahlfors-1974};
  By putting  $\sigma(z)=\frac{\overline{z}}{1-|z|^2}$ in Theorem \ref{thm-1.01}, we obtain the
 result in Theorem \ref{thm-Chen-Que-2017-1} due to Chen and Que \cite{Chen-Que-2017}  and Theorem \ref{thm-Hu-Fan-2021} due to Hu and Fan\cite{Hu-Fan-2021}.
 By setting  $\sigma(z)=\frac{\overline{z}}{1-|z|^2}$ and $\lambda=1$ in Theorem \ref{thm-1.01}, we get  the result just as  Theorem \ref{thm-Hernandez-Martin-2013-1} due to Hern\'{a}ndez and Mart\'{\i}n \cite{Hernandez-Martin-2013}.}
\end{remark}

\begin{theorem}\label{thm-1.02}
 Let $f=h+\overline{g}$ be a sense-preserving harmonic mapping in $\triangle$  with $\|\omega\|_{\infty}<1$. If $f$ satisfies the condition \eqref{eq-1.05},
then   $\widetilde{f}(\partial{\triangle})$ is a quasicircle and  $f$ can be extended to a quasiconformal mapping in $\mathbb{C}$.
\end{theorem}

By means of the pre-Schwazian derivative  and Schwazian derivative    of harmonic
mappings, by setting $$\sigma=v-\frac{1}{2}P_f,$$ and $$\sigma=\frac{c\overline{z}}{1-|z|^2}-\frac{1}{2}P_f$$ in Theorem \ref{thm-1.01}, we get following two corollaries, respectively.

\begin{corollary}\label{coro-1}
For a sense-preserving harmonic mapping   $f$ in $\triangle$, the inequality
$$
\left|\frac{1}{2} S_{f}+v^{2}-v_{z}\right| + \left|(v-P_f)\frac{\omega^{\prime}}{1-|\omega|^2} \right|  \leq k\left|v_{\overline{z}}-\overline{P_f}\right| \ \ (0\leq k<1)
$$
together with $v \rightarrow \infty$ for $|z| \rightarrow 1$ and $\frac{v_{\overline{z}}}{v^{2}} \neq 0$ is sufficient to imply the existence of a $\frac{1+k}{1-k}$-quasiconformal extension of $f$.
\end{corollary}

\begin{corollary}\label{coro-2}
For a sense-preserving harmonic mapping   $f$ in $\triangle$, the inequality
\begin{equation}
\begin{aligned}
\left|\frac{1}{2} S_{f}(1-|z|^2)^2- c(1-c) \overline{z} \right|& + \left|\left(c\overline{z}(1-|z|^2)-\frac{1}{2}P_f(1-|z|^2)^2\right)\frac{\omega^{\prime}}{1-|\omega|^2} \right|\\
  &\qquad  \leq k\left|c-\overline{P_f}(1-|z|^2)^2\right| \ \ (0\leq k<1;\,|c-1|\leq k)
\end{aligned}
\end{equation}
is sufficient to imply the existence of a $\frac{1+k}{1-k}$-quasiconformal extension of $f$.
\end{corollary}

\begin{remark}
{\rm   Corollary \ref{coro-1} and Corollary \ref{coro-2} extended Theorem \ref{thm-Ahlfors-1974-3} and Theorem \ref{thm-Ahlfors-1974-4} to harmonic cases, respectively.}
\end{remark}

\begin{theorem}\label{thm-1.03}
Let $f$ be a sense-preserving harmonic mapping in $\triangle$ with a representation $f=h+\alpha \overline{h}$, where $h$ is a locally univalent analytic function in $\triangle$. Let $\sigma$ satisfy the conditions in Theorem \ref{thm-1.00}  and $\alpha$ be a constant with $|\alpha|<1$. Assume that $h$ satisfies the condition
\begin{equation}\label{eq-1.10}
\left|\sigma P_h+\sigma^{2}-\sigma_{z}\right| \leq k\left|\sigma_{\overline{z}}\right|  \ \ (0\leq k<1).
\end{equation}
Then $f$ is a harmonic Teichm\"{u}ller mappings of $\triangle$ and has a continuous and homeomorphic extension $\tilde{f}$ to $\overline{\triangle}$. Moreover, the mapping
\begin{equation}\label{eq-1.11}
F_{\alpha}(z)=
\begin{cases}\widetilde{f}(z) &  (|z| \leq 1),
\\ f\left(\frac{1}{\overline{z}}\right)+U_{\alpha}\left(\frac{1}{\overline{z}}\right)  & (|z|>1),
\end{cases}
\end{equation}
is a $K$-quasiconformal mapping of the complex plane $\mathbb{C}$ with $K$ given by \eqref{eq-Chen-Que-2017-5},
where
$$
U_{\alpha}(z)=\frac{h^{\prime}(z)}{\sigma(z)}+\alpha\frac{\overline{h^{\prime}(z)}}{\overline{\sigma(z)}}  \quad (z \in \triangle \backslash\{0\}).
$$
The maximal dilatation estimate of the quasiconformal extension $F_{\alpha}$ is asymptotically sharp in $k$, and extremal mappings are of the form $f(z)=m z+n \overline{z}$, where $m$ and $n$ are two nonvanishing constants.
\end{theorem}

\begin{remark}
{\rm By setting  $\sigma(z)=\frac{\overline{z}}{1-|z|^2}$ in Theorem \ref{thm-1.03},  we get the corresponding result obtained by Chen and Que \cite{Chen-Que-2017} (see Theorem \ref{thm-Chen-Que-2017-2}).}
\end{remark}

\section{Proof of  Theorem \ref{thm-1.00}}\label{sec-2}
\begin{proof}[\bf Proof of  Theorem \ref{thm-1.00}]
If $f=h+\overline{g}$ satisfied the condition \eqref{eq-1.04}, we know that
 \begin{align}
\left| {\sigma}  \frac{h^{\prime \prime}}{h^{\prime}}+ {\sigma^{2}} - {\sigma_{z}} \right|
&=\left| {\sigma}\left(\frac{h^{\prime \prime}}{h^{\prime}}-\frac{\overline{\omega} \omega^{\prime}}{1-|\omega|^{2}}\right) + {\sigma^{2}} - {\sigma_{z}} +   \frac{{\sigma}\overline{\omega}{\omega}^{\prime}}{1-|\omega|^{2}}\right|\notag \\
&\leq\left| {\sigma}P_f + {\sigma^{2}} - {\sigma_{z}}\right| + \left| \frac{{\sigma} \overline{\omega}{\omega}^{\prime}}{1-|\omega|^{2}}\right| \\
&\leq \left| {\sigma}P_f + {\sigma^{2}} - {\sigma_{z}}\right| + \left| \frac{{\sigma}{\omega}^{\prime}}{1-|\omega|^{2}}\right|\notag\\
&\leq |\sigma_{\overline{z}}|.\notag
\end{align}
Hence,  from Theorem \ref{thm-Ahlfors-1974-0}, we see that $h$ is univalent.

For $a \in \triangle$, we let  the function $f_{a}$ has the canonical decomposition
\begin{align}
f_{a}= f+a\overline{ f}
=h+{a} g+\overline{g}+a \overline{h}
= h_a+\overline{g_a}\quad(|a|<1).
\end{align}
By noting that $f_{a}$ is also a sense-preserving harmonic mapping and its the second complex dilatation is $\omega_{a}$   with
 $$\omega_{a}=\varphi_{a} \circ \omega,
$$
where $\varphi_{a}$ is the automorphism of $\triangle$ defined by
$$
\varphi_{a}(z)=\frac{\overline{a}+z}{1+a z}  \quad (z \in \triangle).
$$
It follows that
$$
\frac{\left|z \omega_{a}^{\prime}\right|}{1-\left|\omega_{a}\right|^{2}}=\frac{\left|z \varphi_{a}^{\prime}(\omega) \omega^{\prime}\right|}{1-\left|\varphi_{a}(\omega)\right|^{2}}=\frac{\left|z \omega^{\prime}\right|}{1-|\omega|^{2}} \cdot \frac{\left|\varphi_{a}^{\prime}(\omega)\right|\left(1-|\omega|^{2}\right)}{1-\left|\varphi_{a}(\omega)\right|^{2}}=\frac{\left|z \omega^{\prime}\right|}{1-|\omega|^{2}}.
$$
Since the Jacobian   $J_{f_{a}}$ of   $f_{a}$ is $(1-|a|^{2})J_{f}$, we see that $P_{f_{a}}=P_{f}$.
Therefore, $f_a$ satisfies the hypothesis of Theorem \ref{thm-Ahlfors-1974-0}. Then for any $a \in \triangle$, the function $h_a=h+a g$ is univalent. By Hurwitz's theorem, the functions $h+\varepsilon g$ are univalent for all $|\varepsilon|=1$, which implies that $h+\varepsilon\overline{g}$ is univalent (see \cite{Hernandez-Martin-2013-2}). We thus conclude that $f$ is univalent in $\triangle$.
\end{proof}

\vskip.20in
\section{Proof of  Theorem \ref{thm-1.01}}\label{sec-3}

\begin{proof}[\bf Proof of  Theorem \ref{thm-1.01}]
For every $\lambda \in \overline{\triangle}$, let $f_{\lambda}=h+\lambda \overline{g}$.  Then the second complex dilatation of $f_{\lambda}$ is $\omega_{\lambda}=\overline{\lambda} \omega$. Since $(1-|\lambda|)(1-|\omega|)\geq 0$, we obtain

 \begin{align}\label{eq-2.01}
&\left|\sigma P_{f_\lambda}+\sigma^{2}-\sigma_{z}\right| +  \frac{|\sigma \omega_{\lambda}^{\prime}|}{1-|\omega_{\lambda}|^2} \notag\\
&=\left|\sigma\left(\frac{h^{\prime \prime}}{h^{\prime}}-\frac{\overline{\omega_{\lambda}} \omega_{\lambda}^{\prime}}{1-\left|\omega_{\lambda}\right|^{2}}\right)+\sigma^{2}-\sigma_{z}\right|
  +\frac{\left|\sigma\omega^{\prime}_{\lambda}\right|}{1-|\omega_{\lambda}|^2}\notag \\
  &\leq\left|\sigma\left(\frac{h^{\prime \prime}}{h^{\prime}}-\frac{\overline{\omega} \omega^{\prime}}{1-\left|\omega\right|^{2}}\right)+\sigma^{2}-\sigma_{z}\right|+\left|\sigma\left(\frac{\overline{\omega} \omega^{\prime}}{1-\left|\omega\right|^{2}}- \frac{\overline{\omega_{\lambda}} \omega_{\lambda}^{\prime}}{1-\left|\omega_{\lambda}\right|^{2}} \right)\right|
  +\frac{\left|\sigma\omega^{\prime}_{\lambda}\right|}{1-|\omega_{\lambda}|^2}\\
  &=\left|\sigma P_{f}+\sigma^{2}-\sigma_{z}\right|+
  \frac{|\sigma||\omega||\omega^{\prime}|\left(1-|\lambda|^{2}\right)}{\left(1-|\omega|^{2}\right)\left(1-|\lambda|^{2}|\omega|^{2}\right)}
  +\frac{|\sigma||\lambda|\left|\omega^{\prime}\right|}{1-|\lambda|^{2}|\omega|^{2}}\notag \\
&=\left|\sigma P_{f}+\sigma^{2}-\sigma_{z}\right|+
\frac{|\sigma| |\omega^{\prime}| }{1-|\omega|^{2}}\frac{|\lambda|+|\omega|}{1+|\lambda||\omega|}\notag \\
&\leq \left|\sigma P_{f}+\sigma^{2}-\sigma_{z}\right| +  \frac{|\sigma \omega^{\prime}|}{1-|\omega|^2}. \notag
\end{align}
Therefore, for all $\lambda \in \overline{\triangle}$, if  $f$ satisfies the condition \eqref{eq-Ahlfors-1974-1}, we have

$$\left|\sigma P_{f_\lambda}+\sigma^{2}-\sigma_{z}\right| +  \frac{|\sigma \omega_{\lambda}^{\prime}|}{1-|\omega_{\lambda}|^2}\leq k |\sigma_{\overline{z}}|\quad(0\leq k<1).$$

If the second complex dilatation $\omega$   of $f$ satisfies  $\|\omega\|_{\infty}<1$, then for each $\lambda\in \overline{\triangle}$, by \cite[Theorem 1]{Hernandez-Martin-2013}, the mapping $\widetilde{f}_{\lambda}$  can be continuously extended to a homeomorphism
of $\overline{\triangle}$.  Moreover, $F_{\lambda}(z)$ given by \eqref{eq-1.06} is a   homeomorphism  extension  of ${f}_{\lambda}$.

Next, we shall estimate the maximal complex dilatation of the mapping $F_{\lambda}$. We divide it
into the following two cases.

${\bf C{ase}\ 1.}$ \  If $|z|<1$, then $ |\mu_{{F}_{\lambda}} |=|\lambda||\omega(z)| \leq |\lambda|\|\omega\|_{\infty}<1$ for $z\in \triangle$.

${\bf C{ase}\ 2.}$ \   If $|z|>1$, we make a reciprocal transformation  $w=1 / \overline{z}$ for some $\abs{z}< 1$, then
 \begin{align}\label{eq-2.05}
  \left|\mu_{F_\lambda}(w)\right|&=\left|\frac{h^{\prime}(z)+(U_{\lambda})_{z}(z)}{\lambda\overline{g^{\prime}(z)}+(U_{\lambda})_{\overline{z}}(z)}\right|\notag\\
  &=\left|\frac{h^{\prime}+\frac{h^{\prime \prime} \sigma-h^{\prime} \sigma_{z}}{\sigma^{2}}-\lambda\frac{\overline{g^{\prime}} \overline{\sigma_{\overline{z}}}}{\overline{\sigma}^{2}}}{\lambda\overline{g^{\prime}}+\lambda\frac{\overline{g^{\prime \prime}} \overline{\sigma}-\overline{g^{\prime}} \overline{\sigma_{z}}}{\overline{\sigma}^{2}}-\frac{h^{\prime} \sigma_{\overline{z}}}{\sigma^{2}}}\right| \notag\\
  &=\left|\frac{\sigma^{2}+\sigma \frac{h^{\prime \prime}}{h^{\prime}}-\sigma_{z}-\lambda{\overline{\omega}}\overline{\sigma_{\overline{z}}}\frac{\sigma^{2}}{\overline{\sigma}^{2}} \frac{\overline{h^{\prime}}}{h^{\prime}}}{\overline{\lambda}\sigma^{2} \omega+\overline{\lambda}\sigma\frac{g^{\prime \prime}}{h^{\prime}}- \overline{\lambda}\sigma_{z}\omega-\overline{{\sigma_{\overline{z}}}}\frac{\sigma^{2}}{\overline{\sigma}^{2}} \frac{\overline{h^{\prime}}}{h^{\prime}}}\right|\\
  &=\left|\frac{\frac{\sigma^{2}}{\sigma_{\overline{z}}}+\frac{\sigma}{\sigma_{\overline{z}}} \frac{h^{\prime \prime}}{h^{\prime}} -\frac{\sigma_{z}}{\sigma_{\overline{z}}}-\lambda\overline{\omega}\frac{\sigma^{2}}{\overline{\sigma}^{2}} \frac{\overline{\sigma_{\overline{z}}}}{\sigma_{\overline{z}}}\frac{\overline{h^{\prime}}}{h^{\prime}}}{\overline{\lambda}\frac{\sigma^{2}}{\sigma_{\overline{z}}} \omega+\overline{\lambda}\frac{\sigma}{\sigma_{\overline{z}}} \frac{g^{\prime \prime}}{h^{\prime}}-\overline{\lambda}\frac{\sigma_{z}}{\sigma_{\overline{z}}} \omega-\frac{\sigma^{2}}{\overline{\sigma}^{2}} \frac{\overline{\sigma_{\overline{z}}}}{\sigma_{\overline{z}}}\frac{\overline{h^{\prime}}}{h^{\prime}}}\right| \notag \\
&\leq \frac{| \frac{\sigma^{2}}{\sigma_{\overline{z}}}+\frac{\sigma}{\sigma_{\overline{z}}} \frac{h^{\prime \prime}}{h^{\prime
}}-\frac{\sigma_{z}}{\sigma_{\overline{z}}}|+|\lambda|\|\omega\|_{\infty}}{1-|\lambda|\left|\frac{\sigma^{2}}{\sigma_{\overline{z}}} \omega-\frac{\sigma_{z}}{\sigma_{\overline{z}}} \omega+\frac{\sigma}{\sigma_{\overline{z}}} \frac{g^{\prime \prime}}{h^{\prime }}\right|}.\notag
  \end{align}
On the one hand, since $f$ satisfies the condition \eqref{eq-1.05}, by the triangle inequality, we get

\begin{align}\label{eq-2.06}
\left|\frac{\sigma^{2}}{\sigma_{\overline{z}}}+\frac{\sigma}{\sigma_{\overline{z}}} \frac{h^{\prime \prime}}{h^{\prime}}-\frac{\sigma_{z}}{\sigma_{{\overline{z}}}}\right|
&=\left|\frac{\sigma}{\sigma_{\overline{z}}}\left(\frac{h^{\prime \prime}}{h^{\prime}}-\frac{\overline{\omega} \omega^{\prime}}{1-|\omega|^{2}}\right)+\frac{\sigma^{2}}{\sigma_{\overline{z}}}-\frac{\sigma_{z}}{\sigma_{\overline{z}}}+\frac{\sigma}{\sigma_{\overline{z}}} \frac{\overline{\omega}{\omega}^{\prime}}{(1-|\omega|^{2})}\right|\notag \\
&\leq\left|\frac{\sigma}{\sigma_{\overline{z}}}\left(\frac{h^{\prime \prime}}{h^{\prime}}-\frac{\overline{\omega} \omega^{\prime}}{1-|\omega|^{2}}\right)+\frac{\sigma^{2}}{\sigma_{\overline{z}}}-\frac{\sigma_{z}}{\sigma_{\overline{z}}}\right|+\left|\frac{\sigma}{\sigma_{\overline{z}}} \frac{\overline{\omega}{\omega}^{\prime}}{(1-|\omega|^{2})}\right|\notag\\
&\leq \left|\frac{\sigma}{\sigma_{\overline{z}}}P_f+\frac{\sigma^{2}}{\sigma_{\overline{z}}}-\frac{\sigma_{z}}{\sigma_{\overline{z}}}\right|+\left|\frac{\sigma}{\sigma_{\overline{z}}} \frac{\overline{\omega}{\omega}^{\prime}}{(1-|\omega|^{2})}\right|\\
&\leq  k-\left|\frac{\sigma}{\sigma_{\overline{z}}} \frac{{\omega}^{\prime}}{1-|\omega|^{2}}\right|+\|\omega\|_{\infty}\left|\frac{\sigma}{\sigma_{\overline{z}}} \frac{{\omega}^{\prime}}{(1-|\omega|^{2})}\right|\notag\\
&=k-\left(1-\|\omega\|_{\infty}\right)\left|\frac{\sigma}{\sigma_{\overline{z}}} \frac{{\omega}^{\prime}}{(1-|\omega|^{2})}\right|.\notag
\end{align}
On the other hand, we have $$g^{\prime \prime}=\left(\omega h^{\prime}\right)^{\prime}=\omega^{\prime} h^{\prime}+\omega h^{\prime \prime}.$$ Therefore, for all $z \in \triangle$, we obtain

\begin{align}\label{eq-2.07}
\left|\frac{\sigma^{2}}{\sigma_{\overline{z}}} \omega-\frac{\sigma_{z}}{\sigma_{\overline{z}}} \omega+\frac{\sigma}{\sigma_{\overline{z}}} \frac{g^{\prime \prime}}{h^{\prime}}\right|
&=\left|\frac{\sigma^{2}}{\sigma_{\overline{z}}} \omega-\frac{\sigma_{z}}{\sigma_{\overline{z}}} \omega+\frac{\sigma}{\sigma_{\overline{z}}} \frac{h^{\prime \prime}}{h^{\prime}}+\frac{\sigma}{\sigma_{\overline{z}}}\omega^{\prime}\right|\notag\\
&=\left|\frac{\sigma}{\sigma_{\overline{z}}}\left(\frac{h^{\prime \prime}}{h^{\prime}}-\frac{\overline{\omega} \omega^{\prime}}{1-|\omega|^{2}}\right)+\frac{\sigma^{2}}{\sigma_{\overline{z}}}\omega-
\frac{\sigma_{z}}{\sigma_{\overline{z}}}\omega+\frac{\sigma}{\sigma_{\overline{z}}} \frac{{\omega}^{\prime}}{(1-|\omega|^{2})}\right|\notag\\
&\leq \|\omega\|_{\infty}\left|\frac{\sigma}{\sigma_{\overline{z}}} P_{f}+ \frac{\sigma^{2}}{\sigma_{\overline{z}}}-\frac{\sigma_{z}}{\sigma_{\overline{z}}} \right|+\left|\frac{\sigma}{\sigma_{\overline{z}}} \frac{\omega^{\prime}}{(1-|\omega|^{2})}\right|\\
&\leq  k\|\omega\|_{\infty}-\|\omega\|_{\infty}\left|\frac{\sigma}{\sigma_{\overline{z}}} \frac{\omega^{\prime}}{(1-|\omega|^{2})}\right|+\left|\frac{\sigma}{\sigma_{\overline{z}}} \frac{\omega^{\prime}}{(1-|\omega|^{2})}\right|\notag\\
&=k\|\omega\|_{\infty}+\left(1-\|\omega\|_{\infty}\right)\left|\frac{\sigma}{\sigma_{\overline{z}}} \frac{\omega^{\prime}}{(1-|\omega|^{2})}\right|.\notag
\end{align}
By substituting \eqref{eq-2.06} and \eqref{eq-2.07} into \eqref{eq-2.05}, we know that
\begin{equation}\label{eq-2.08}
\begin{aligned}
\left|\mu_{F_\lambda}(w)\right| & \leq \frac{k-\left(1-\|\omega\|_{\infty}\right)\left|\omega^{*}(z)\right|+|\lambda|\|\omega\|_{\infty}}{1-\left(k|\lambda|\|\omega\|_{\infty}+
\left(1-\|\omega\|_{\infty}\right)\left|\omega^{*}(z)\right|\right)} \\
&=\frac{k+|\lambda|\|\omega\|_{\infty}-\left(1-|\lambda|\|\omega\|_{\infty}\right)\left|\omega^{*}(z)\right|}{1-k|\lambda|\|\omega\|_{\infty}-
\left(1-\|\omega\|_{\infty}\right)\left|\omega^{*}(z)\right|},
\end{aligned}
\end{equation}
where $$\omega^{*}(z)=\frac{\sigma}{\sigma_{\overline{z}}} \frac{\omega^{\prime}}{(1-|\omega|^{2})}.$$ Since $f$ satisfies the condition \eqref{eq-1.05}, we deduce that $|\omega^{*}(z)|\leq k$.

Define the function $\rho(x):[0, k] \rightarrow \mathbb{R}$ by
$$
\rho(x)=\frac{k+|\lambda|\|\omega\|_{\infty}-\left(1-\|\omega\|_{\infty}\right) x}{1-k|\lambda|\|\omega\|_{\infty}-\left(1-\|\omega\|_{\infty}\right) x}.
$$
Bearing in mind that \eqref{eq-Hernandez-Martin-2013-3} holds, we see that $\rho^{\prime}(x)<0$ for all $x \in[0, k].$ Thus,
$$
\rho(x) \leq \rho(0)=\frac{k+|\lambda|\|\omega\|_{\infty}}{1-k|\lambda|\|\omega\|_{\infty}}=\frac{K-1}{K+1}.
$$
We then find from \eqref{eq-2.08} that the inequality
$$\left|\mu_{F_\lambda}(w)\right| \leq \frac{K-1}{K+1}:=k_1$$
holds for all $|w|>1$. Since the assumption \eqref{eq-Hernandez-Martin-2013-3} shows that $$\left\|\omega_{\lambda}\right\|_{\infty}=|\lambda|\|\omega\|_{\infty}\leq \|\omega\|_{\infty}<\frac{1-k}{1+k}$$
 holds for all $\lambda \in \overline{\triangle}$,
 which implies that $k_{1}<1$. Therefore, $F_{\lambda}(w)$ is a quasiconformal mapping in $\mathbb{C} \backslash \overline{\triangle}$.

By  \cite[Lemma 6.1]{Lehto-1987}, we see that $$\left|\mu_{F_\lambda}(z)\right| \leq\frac{K-1}{K+1}$$ for all $|z| \neq 1$ in $\mathbb{C}$. Thus, the mapping $F_\lambda$ defined by \eqref{eq-1.06} is $K$-quasiconformal whenever \eqref{eq-2.08} holds. Moreover, by noting that $|\lambda|\|\omega\|_{\infty}\leq k_1<1$, we deduce that $F_{\lambda}$ is a $K$-quasiconformal mapping of $f_{\lambda}$ in $\mathbb{C}$, where $K$ is given by \eqref{eq-Chen-Que-2017-2}.
\end{proof}

\begin{remark}\label{rem-thm-1.01}
{\rm The method  to prove   $F_{\lambda}(z)$ given by \eqref{eq-1.06} is a    homeomorphism  extension  of $\mathbb{C}$ due to \cite{Hernandez-Martin-2013}. Here, we give the main key points.

 Firstly, we shall prove that for all $a \in \overline{\triangle}$, the functions $h_{a}=h+a g$ have a continuous and injective extension to $\mathbb{C}$. It follows from \eqref{eq-1.05} that for all $z \in \triangle$,
\begin{equation}\label{eq-2.04}
 \begin{aligned}
k|\sigma_{\overline{z}}| & \geq \left|\sigma P_{f_a}+\sigma^{2}-\sigma_{z}\right| +  \frac{|\sigma \omega_a^{\prime}(z)|}{1-|\omega_a(z)|^2} \\
  &=\left|\sigma\left(\frac{h^{\prime \prime}(z)+a g^{\prime \prime}(z)}{h^{\prime}(z)+a g^{\prime}(z)}-\frac{\overline{\omega_{a}(z)} \omega_{a}^{\prime}(z)}{1-\left|\omega_{a}(z)\right|^{2}}\right)+\sigma^{2}-\sigma_{z}\right|
  +\frac{\left|\sigma\omega^{\prime}_{a}(z)\right|}{1-|\omega_a(z)|^2}.
\end{aligned}
\end{equation}
In view of \eqref{eq-2.04} and the triangle inequality,  we see that
\begin{equation}\label{eq-2.02}
\begin{aligned}
\left|\sigma \frac{h_a^{\prime \prime}(z)}{h_a^{\prime}(z)}+\sigma^{2}-\sigma_{z}\right|
\leq
  &\left|\sigma\left(\frac{h^{\prime \prime}(z)+a g^{\prime \prime}(z)}{h^{\prime}(z)+a g^{\prime}(z)}-\frac{\overline{\omega_{a}(z)} \omega_{a}^{\prime}(z)}{1-\left|\omega_{a}(z)\right|^{2}}\right)+\sigma^{2}-\sigma_{z}\right|
  +\left|\frac{\sigma\overline{\omega_{a}(z)}\omega^{\prime}_{a}(z)}{1-|\omega_{a}(z)|^2}\right| \\
\leq &\left|\sigma P_{f_{a}}(z)+\sigma^{2}-\sigma_{z}  \right|+\left|\frac{\sigma\omega^{\prime}_{a}(z)}{1-|\omega_a(z)|^2}\right|  \\ \leq & k |\sigma_{\overline{z}}|.
\end{aligned}
\end{equation}
By virtue of \eqref{eq-Ahlfors-1974-1}, it implies that for each $a \in \triangle$,
the function $h_{a}$ is univalent and can be extended to a continuous and injective mapping $\widetilde{h}_{a}$ in $\overline{\triangle}$. Moreover, in \cite{Ahlfors-1974}, it was shown that the function
\begin{equation}\label{eq-2.03}
H_{a}(z)=
 \begin{cases}
 \widetilde{h}_{a}(z) & (|z| \leq 1), \\
  h_{a}\left(\frac{1}{\overline{z}}\right)+u_{a}\left(\frac{1}{\overline{z}}\right) & (|z|>1), \ \
 \end{cases}
\end{equation}
is a $K$-quasiconformal extension of $\mathbb{C}$ onto itself with $K=(1+k) /(1-k)$, where $$u_{a}(z)=\frac{h_{a}^{\prime}(z) }{{\sigma(z)}} \quad (z\in \triangle \backslash\{0\}).$$ Hence,  $H_{a}$ is continuous and univalent in $\mathbb{C}$.

Define
$$
H_{\varepsilon}(z)=\lim _{r \rightarrow 1^{-}} H_{r e^{i \theta}}(z)  \quad (z \in \mathbb{C};\,\varepsilon=e^{i \theta} \in \partial \triangle).$$
 According to \cite[Theorem 1]{Hernandez-Martin-2013}, we show that $H_{\varepsilon}$ is continuous and one-to-one in $\mathbb{C}$ for all $|\varepsilon|\leq 1$.

Next,  we assume that $H=H_{0}$ and $H_{1}$ are the corresponding extension of $h_{0}=h$ and $h_{1}=h+g$, respectively. By virtue of \eqref{eq-2.03},
we define  the function
$$
G(z)=H_{1}(z)-H(z)
$$
 to obtain a continuous extension $G$ of $g$ to $\mathbb{C}$, which was given by
\begin{equation*}
G(z)= \begin{cases}
       g(z)  & (|z|<1), \\
        \widetilde{g}(z)=\widetilde{h}_{1}(z)-\widetilde{h}(z)  & (|z|=1), \\ g\left(\frac{1}{\overline{z}}\right)+V\left(\frac{1}{\overline{z}}\right)  & (|z|>1),
       \end{cases}
\end{equation*}
where
$$
V(z)=\frac{g^{\prime}(z)}{\sigma(z)}  \quad (z \in \triangle \backslash\{0\}).
$$
Since $H_{a}$ are univalent in $\mathbb{C}$ for all $a \in \overline{\triangle}$, for each $z \in \mathbb{C}$,  we see that $H(z)+a G(z)$ is univalent  in $\mathbb{C}$.
Now,  we construct an explicit candidate for a continuous and injective extension of $f$ to $\mathbb{C}$  defined by
\begin{equation*}
F_{\lambda}(z)=H(z)+\lambda\overline{G(z)}=
    \begin{cases}
           f_{\lambda}(z)  & (|z|<1), \\
           \widetilde{h}(z)+\lambda\overline{\widetilde{g}}(z)  & (|z|=1), \\ f_{\lambda}\left(\frac{1}{\overline{z}}\right)+U\left(\frac{1}{\overline{z}}\right)  & (|z|>1),
    \end{cases}
\end{equation*}
where $U$ is given by
$$
U(z)=\frac{h^{\prime}(z)}{\sigma(z)}+\lambda\frac{\overline{g^{\prime}(z)}}{\overline{\sigma(z)}}\quad(z\in\triangle \backslash\{0\}).
$$

Finally,
by using the similar method
as in \cite{Hernandez-Martin-2013}, we can prove that $F_{\lambda}$ is univalent in $\mathbb{C}$, and $F_{\lambda}$ is a homeomorphism of $\mathbb{C}$ onto itself.}

\end{remark}

\vskip.20in
\section{Proof of Theorem \ref{thm-1.02}}\label{sec-4}

The following three lemmas are crucial to prove Theorem \ref{thm-1.02}.
\begin{lemma}\label{lem2.03}
{\rm (}\cite{Chen-Que-2017}{\rm )} Let $\varepsilon \in \triangle$ and $T(z)=(z+|\varepsilon|) /(1+|\varepsilon| z)$. Then $T(z)$ is a M\"{o}bius transformation of the unit disk $\triangle$ onto itself and
$$
\frac{\|\varepsilon|-| z||}{1-|\varepsilon||z|} \leq|T(z)| \leq \frac{|\varepsilon|+|z|}{1+|\varepsilon||z|}.
$$
\end{lemma}

\begin{lemma}\label{lem2.01}
Let $f=h+\overline{g}$ be a sense-preserving harmonic mapping in $\triangle$ with complex dilatation $\omega \not \equiv 0.$ Assume that $\|\omega\|_{\infty}<1$ and that $f$ satisfies \eqref{eq-1.05}. Then, the analytic functions $h_{a}=h+a g$ are univalent in $\triangle$ for all $0 \leq|a|<\delta$, where
\begin{equation*}\label{eq-2.09}
1<\delta=\frac{1+k\|\omega\|_{\infty}}{k+\|\omega\|_{\infty}} \leq \frac{1}{\|\omega\|_{\infty}}.
\end{equation*}
Moreover, $h_{a}$ has a continuous and injective extension $\widetilde{h}_{a}$ to $\overline{\triangle}$.
\end{lemma}
\begin{proof}
In  Remark \ref{rem-thm-1.01}, we have found that $h_{a}$ has a continuous and injective extension $\widetilde{h}_{a}$ to $\overline{\triangle}$ for all $|a| \leq 1$. Next, we prove that $h_{a}$ has a continuous and injective extension $\widetilde{f_{a}}$ to $\overline{\triangle}$ for $1<|a|<\delta$.
By noting that
\begin{equation}\label{eq-2.10}
\frac{h_{a}^{\prime \prime}}{h_{a}^{\prime}}=\frac{h^{\prime \prime}}{h^{\prime}}+\frac{a \omega^{\prime}}{1+a \omega},
\end{equation}
in view of \eqref{eq-2.10} and the formula \eqref{eq-Pf} for pre-Schwarzian derivative of $f$,  we get
\begin{equation}\label{eq-2.11}
\begin{aligned}
\left|\frac{\sigma}{\sigma_{\overline{z}}}   \frac{h_{a}^{\prime \prime}}{h_{a}^{\prime}}+\frac{\sigma^{2}}{\sigma_{\overline{z}}}-\frac{\sigma_{z}}{\sigma_{\overline{z}}}\right| &\leq\left|\frac{\sigma}{\sigma_{\overline{z}}} P_{f}+\frac{\sigma^{2}}{\sigma_{\overline{z}}}-\frac{\sigma_{z}}{\sigma_{\overline{z}}}\right|+\left|\frac{\sigma}{\sigma_{\overline{z}}} \frac{\omega^{\prime}}{1-|\omega|^2}  \frac{\overline{\omega}+a}{1+a \omega}\right|\\
&\leq  k-\left|\frac{\sigma}{\sigma_{\overline{z}}}  \frac{\omega^{\prime}}{1-|\omega|^{2}}\right|+\left|\frac{\sigma}{\sigma_{\overline{z}}}  \frac{\omega^{\prime}}{1-|\omega|^{2}}\right|  \left|\frac{\overline{\omega}+a}{1+a \omega}\right|\\
&=k+  \left|\frac{\sigma}{\sigma_{\overline{z}}}   \frac{\omega^{\prime}}{1-|\omega|^{2}}\right|\left( \left|\frac{\overline{\omega}+a}{1+a \omega}\right|-1\right)\\
&\leq  k\left|\frac{\overline{\omega}+a}{1+a \omega}\right|.
\end{aligned}
\end{equation}
From Lemma \ref{lem2.03}, we have
\begin{equation}\label{eq-2.12}
\sup _{z \in \mathbf{D}}\left|\frac{\omega+\bar{a}}{1+a \omega}\right| \leq \frac{|a|-\|\omega\|_{\infty}}{1-\|\omega\|_{\infty}|a|}<\frac{1}{k}
\end{equation}
for $$1<|a|<\delta=\frac{1+k\|\omega\|_{\infty}}{k+\|\omega\|_{\infty}} \leq \frac{1}{\|\omega\|_{\infty}}.$$
Combining \eqref{eq-2.11}, \eqref{eq-2.12} with Theorem \ref{thm-Ahlfors-1974-1}, we conclude that $h_{a}$ is univalent in $\triangle$ and $h_{a}$ has a continuous and injective extension $\widetilde{f}_{a}$ to $\overline{\triangle}$ for all $0 \leq|a|<\delta$, where $$1<\delta=\frac{1+k\|\omega\|_{\infty}}{k+\|\omega\|_{\infty}} \leq\frac{1}{\|\omega\|_{\infty}}.$$
We thus complete the proof of Lemma \ref{lem2.01}.
\end{proof}

\begin{lemma}\label{lem2.02}
 Let $f=h+\overline{g}$ satisfy the hypothesis of Lemma \ref{lem2.01}. Assume in addition that both $h$ and $g$ are analytic functions in $\overline{\triangle}$. Then,
\begin{equation}\label{eq-2.13}
\sup _{\alpha, \beta \in \overline{\triangle}, \alpha \neq \beta}\left|\frac{g(\alpha)-g(\beta)}{h(\alpha)-h(\beta)}\right|
=\sup _{\alpha, \beta \in \overline{\triangle}, \alpha \neq \beta}\left|\frac{\lambda\overline{(g(\alpha)-g(\beta))}}{h(\alpha)-h(\beta)}\right|  \leq \frac{1}{\delta}=\frac{k+\|\omega\|_{\infty}}{1+k\|\omega\|_{\infty}}<1.
\end{equation}
\end{lemma}

\begin{proof}
By using the similar method as in \cite[Lemma 2]{Hu-Fan-2021}, we can prove Lemma \ref{lem2.02}. The details are omitted here.
\end{proof}

\begin{proof}[\bf Proof of  Theorem \ref{thm-1.02}]
Define   $$(f_{\lambda})_r(z)=f_{\lambda}(r z)=h(r z)+\lambda\overline{g(r z)}\quad (0<r<1),$$ where both $h$ and $g$ are analytic in $\overline{\triangle}$.
Each of these functions $f_{\lambda}$ with $|\lambda||\omega_{r}|=|\lambda||\omega(r z)|$, so that $$|\lambda|\|\omega_{r}\|_{\infty} =\|\omega_{r}\|_{\infty}\leq\|\omega\|_{\infty}<1.$$

If we   prove that the mapping $(f_{\lambda})_r$ can be extended to a $K$-quasiconformal mapping in $\mathbb{C}$, where $K$ does not depend on $r$. Then  from \cite[Theorem 5.3]{Lehto-Virtanen-1973},  we conclude that $f_{\lambda}$ can be   extended to a $K$-quasiconformal mappings in $\mathbb{C}$.

We note that if $f$ satisfies \eqref{eq-1.05}, then
\begin{equation*}
\begin{aligned}
\sup _{z \in \triangle}\left\{|\sigma P_{(f_{\lambda})_r}+\sigma^{2}-\sigma_{z} | +  \frac{|\sigma \omega^{\prime}|}{1-|\omega|^2}\right\}
& \leq r  \sup _{z \in \triangle}|\sigma P_{f_{\lambda}}+\sigma^{2}-\sigma_{z} |+r\sup _{z \in \triangle}\frac{|\sigma \omega^{\prime}|}{1-|\omega|^2}\\
&\leq r  \sup _{z \in \triangle}|\sigma P_{f}+\sigma^{2}-\sigma_{z} |+r\sup _{z \in \triangle}\frac{|\sigma \omega^{\prime}|}{1-|\omega|^2}\\
& \leq r k|\sigma_{\overline{z}}|\leq  k|\sigma_{\overline{z}}|.
\end{aligned}
\end{equation*}
Thus, $(f_{\lambda})_r$ satisfies the condition \eqref{eq-1.05}.

By Lemmas \ref{lem2.01} and \ref{lem2.02}, and use the similar method as in \cite[Theorem 3]{Hu-Fan-2021}, we get the assertion $(f_{\lambda})_r$  admits a $K$-quasiconformal reflection. By allowing $r \rightarrow 1$,   we deduce that  $f_{\lambda}$ has a
quasiconformal extension to $\mathbb{C}$. This completes the proof of Theorem \ref{thm-1.02}.
\end{proof}

\vskip.20in
\section{Proof of Theorem \ref{thm-1.03}}\label{sec-5}
\begin{proof}[\bf Proof of  Theorem \ref{thm-1.03}]
{\bf Step\ 1.}  We shall prove that for any $\alpha$ with $|\alpha|<1, f=h+\alpha \overline{h}$ is a harmonic Teichm\"{u}ller mapping in $\triangle$.

Assume that $f$ is a harmonic mapping in  $\triangle$ with a representation $f(z)=h(z)+\alpha \overline{h(z)}$, where $h$ is a locally univalent analytic function and $|\alpha|<1$. It follows that
$$
\frac{|\sigma \omega^{\prime}|}{1-|\omega|^2}=0.
$$
The Jacobian $J_{f}$ of $f$ is equal to $\left(1-|\alpha|^{2}\right) J_{f}$.
Hence,
\begin{equation}\label{eq-2.14}
P_{f}=\frac{\partial}{\partial z} \log \left(\left(1-|a|^{2}\right)\left|h^{\prime}(z)\right|^{2}\right)=\frac{h^{\prime\prime}}{h^{\prime}}=P_{h}.
\end{equation}
By \eqref{eq-Ahlfors-1974-0} and \eqref{eq-Ahlfors-1974-1}, the  Ahlfors's univalent and  extension criterion for a locally univalent analytic function shows that $h$ is univalent, and it can be extended to a continuous and injective mapping $\widetilde{h}(z)$ in $\overline{\triangle}$. Thus, for any $\alpha$ with $|\alpha|<1, f=h+\alpha \overline{h}$ is a harmonic Teichm\"{u}ller mapping in $\triangle$.

{\bf Step\ 2.} We will show that ${F}(z)$  is a homeomorphic extension of $f$ in $\mathbb{C}$ onto itself.
By virtue of \cite[Theorem 1]{Hernandez-Martin-2013} and \eqref{eq-2.14}, we have
\begin{equation}\label{eq-2.15}
\left|\sigma P_f+\sigma^{2}-\sigma_{z}\right| \leq k\left|\sigma_{\overline{z}}\right|.
\end{equation}
 Hence, $f=h+\alpha \overline{h}$ can be extended to a homeomorphism of $\mathbb{C}$ onto itself. Moreover, the homeomorphism can be constructed as $F_{\alpha}(z)$ given by \eqref{eq-1.11}.

{\bf Step\ 3.} We shall estimate the maximal dilatation of $F_{\alpha}$, and then show that $F_{\alpha}$ is a quasiconformal mapping in $\mathbb{C}$.
By the assumption $|\alpha|<1$, we now separate the argument into two cases:

{\bf Case\ 1.}  If $|z|<1$, then $\left|\mu_{F_{\alpha}}(z)\right|=|\alpha|<1.$

{\bf Case\ 2.} If $|z|>1$, we make a reciprocal transformation  $\zeta=1 / \overline{z}$ for $|z|<1$, then
\begin{align}\label{eq-uf}
\left|\mu_{F_{\alpha}}(\zeta)\right|&=\left|\frac{(F_{\alpha})_{\overline{\zeta}}}{(F_{\alpha})_{\zeta}}\right|
=\left|\frac{(F_{\alpha})_{z}  z_{\overline{\zeta}}+(F_{\alpha})_{\overline{z}} \overline{z_{\zeta}}}{(F_{\alpha})_{z}z_{\zeta}+(F_{\alpha})_{\overline{z}}  \overline{z_{\overline{\zeta}}}}\right|=\left|\frac{(F_{\alpha})_{z}}{(F_{\alpha})_{\overline{z}}}\right|\notag\\
  &=\left|\frac{h^{\prime}+\frac{h^{\prime \prime} \sigma-h^{\prime} \sigma_{z}}{\sigma^{2}}-\alpha\frac{\overline{h^{\prime}} \overline{\sigma_{\overline{z}}}}{\overline{\sigma}^{2}}}{\alpha\overline{h^{\prime}}+\alpha\frac{\overline{h^{\prime \prime}} \overline{\sigma}-\overline{h^{\prime}} \overline{\sigma_{z}}}{\overline{\sigma}^{2}}-\frac{h^{\prime} \sigma_{\overline{z}}}{\sigma^{2}}}\right| \notag\\
  &=\left|\frac{\sigma^{2}+\sigma \frac{h^{\prime \prime}}{h^{\prime}}-\sigma_{z}-\alpha\overline{\sigma_{\overline{z}}}\frac{\sigma^{2}}{\overline{\sigma}^{2}} \frac{\overline{h^{\prime}}}{h^{\prime}}}{\overline{\alpha}\sigma^{2} +\overline{\alpha}\sigma\frac{h^{\prime \prime}}{h^{\prime}}- \overline{\alpha}\sigma_{z}-\overline{{\sigma_{\overline{z}}}}\frac{\sigma^{2}}{\overline{\sigma}^{2}}\frac{\overline{h^{\prime}}}{h^{\prime}} }\right| \\
  &=\left|\frac{\frac{\sigma^{2}}{\sigma_{\overline{z}}}+\frac{\sigma}{\sigma_{\overline{z}}} \frac{h^{\prime \prime}}{h^{\prime}} -\frac{\sigma_{z}}{\sigma_{\overline{z}}}-\alpha\frac{\sigma^{2}}{\overline{\sigma}^{2}} \frac{\overline{\sigma_{\overline{z}}}}{\sigma_{\overline{z}}}\frac{\overline{h^{\prime}}}{h^{\prime}}}{\overline{\alpha}\frac{\sigma^{2}}{\sigma_{\overline{z}}} +\overline{\alpha}\frac{\sigma}{\sigma_{\overline{z}}} \frac{h^{\prime \prime}}{h^{\prime}}-\overline{\alpha}\frac{\sigma_{z}}{\sigma_{\overline{z}}} -\frac{\sigma^{2}}{\overline{\sigma}^{2}} \frac{\overline{\sigma_{\overline{z}}}}{\sigma_{\overline{z}}}\frac{\overline{h^{\prime}}}{h^{\prime}}}\right| \notag\\
  &=\left|\frac{\alpha+\frac{-\frac{\sigma^{2}}{\sigma_{\overline{z}}}
-\frac{\sigma}{\sigma_{\overline{z}}} \frac{h^{\prime \prime}}{h^{\prime}}
+\frac{\sigma_{z}}{\sigma_{\overline{z}}}}{\frac{\sigma^{2}}{\overline{\sigma}^{2}} \frac{\overline{\sigma_{\overline{z}}}}{\sigma_{\overline{z}}}\frac{\overline{h^{\prime}}}{h^{\prime}}}}{1+\overline{\alpha} \frac{-\frac{\sigma^{2}}{\sigma_{\overline{z}}}
-\frac{\sigma}{\sigma_{\overline{z}}} \frac{h^{\prime \prime}}{h^{\prime}}
+\frac{\sigma_{z}}{\sigma_{\overline{z}}}}{\frac{\sigma^{2}}{\overline{\sigma}^{2}} \frac{\overline{\sigma_{\overline{z}}}}{\sigma_{\overline{z}}}\frac{\overline{h^{\prime}}}{h^{\prime}}}}\right|.\notag
\end{align}

We now suppose that
$$\eta(z)= \frac{-\frac{\sigma^{2}}{\sigma_{\overline{z}}}
-\frac{\sigma}{\sigma_{\overline{z}}} \frac{h^{\prime \prime}}{h^{\prime}}
+\frac{\sigma_{z}}{\sigma_{\overline{z}}}}{\frac{\sigma^{2}}{\overline{\sigma}^{2}} \frac{\overline{\sigma_{\overline{z}}}}{\sigma_{\overline{z}}}\frac{\overline{h^{\prime}}}{h^{\prime}}}.
$$
From \eqref{eq-1.05}, we see that
\begin{equation}\label{eq-2.16}
|\eta(z)| \leq k \quad(z \in \triangle).
\end{equation}
By setting $\theta=\overline{\alpha} /|\alpha|$, it follows from \eqref{eq-uf}  that
$$
\left|\mu_{F_{\alpha}}(\zeta)\right|=\left|\frac{\alpha+\eta(z)}{1+\overline{\alpha} \eta(z)}\right|=\left|\frac{\theta \alpha+\theta \eta(z)}{1+\overline{\theta}\overline{\alpha} \theta \eta(z)}\right|=\left|\frac{|\alpha|+\theta \eta(z)}{1+|\alpha| \theta \eta(z)}\right|  \quad (\zeta \in \mathbb{C}\backslash \overline{\triangle}).
$$
By Lemma \ref{lem2.03} and \eqref{eq-2.16}, we know that
\begin{equation}\label{eq-2.17}
\left|\mu_{F_{\alpha}}(\zeta)\right| \leq \frac{|\alpha|+|\eta(z)|}{1+|\alpha \eta(z)|} \leq \frac{|\alpha|+k}{1+|\alpha| k}=k_{2} \quad (\zeta \in \mathbb{C} \backslash \overline{\triangle}).
\end{equation}
Therefore, $F_{\alpha}(\zeta)$ is a quasiconformal mapping in $\mathbb{C}\backslash \overline{\triangle}$.

By taking \begin{equation}\label{eq-2.18}
 f=m z+n \overline{z}\quad (|n|<|m|),
 \end{equation}
then $f$ satisfies the condition \eqref{eq-1.10} with $k=0$. Its quasiconformal extension given by \eqref{eq-1.11} is just of the form \eqref{eq-2.18}, and its maximal dilatation is $$k_{2}=\left|\frac{\overline{n}}{m}\right|=|\alpha|=\frac{K-1}{K+1},$$ which shows that the estimate \eqref{eq-2.17} is asymptotically sharp for $k$. The proof of Theorem \ref{thm-1.03} is thus completed.
\end{proof}




\vskip .20in
\begin{center}{\sc Acknowledgments}
\end{center}

\vskip.01in
The present investigation was supported by the \textit{Key Project of Education Department of Hunan Province} under Grant no.
19A097, and the \textit{Natural Science Foundation of Hunan Province} under Grant no. 2018JJ2074 of the P. R. China.

\end{document}